\documentclass[11pt]{amsart}
\begin{document}
%\pages{1}{7}
%%%%% USER-DEFINED MACROS HERE %%%%%

%\newtheorem{definition}{Definition}[section]
%\newtheorem{theorem}[definition]{Theorem}
%\newtheorem{corollary}[definition]{Corollary}
%\newtheorem{lemma}[definition]{Lemma}
%\newtheorem{proposition}[definition]{Proposition}
%\newtheorem{notation}[definition]{Notation}

%Please do not alter the page dimensions,
%the text font size or running head
%%%%%%%%%%%%%%%%%%%%%%%%%%%%%%%%%%%
%example of font substitution
  \def\Bbb#1{\mbox{\sf #1}} % to be replaced with blackboard bold 
%%%%%%%%%%%%%%%%%%%
%%%%% For multiletter symbols %%%%%
%
\def\Real{\mathop{\rm Re}}      % cf plain TeX's \Re and Reynolds number
\def\Imag{\mathop{\rm Im}}      % cf plain TeX's \Im
%
%%%%%

\newtheorem{definition}{Definition}[section]
\newtheorem{theorem}[definition]{Theorem}
\newtheorem{corollary}[definition]{Corollary}
\newtheorem{lemma}[definition]{Lemma}
\newtheorem{proposition}[definition]{Proposition}
\newtheorem{notation}[definition]{Notation}
%There follows some examples of abreviations of common macros
%\def\uc#1{{\uppercase{#1}}}
%\newtheorem{theorem}{T{\sc HEOREM}}[section]
%\newtheorem{lemma}[theo]{L{\sc EMMA}}
%\newtheorem{proposition}[theo]{P{\sc ROPOSITION}}
%\newtheorem{corollary}[theo]{C{\sc OROLLARY}}
%\def\bth{\begin{theorem}}
%\def\eth{\end{theorem}}
%\def\be{\begin{equation}}
%\def\ee{\end{equation}}
%\def\bprop{\begin{prop}}
%\def\eprop{\end{prop}}
%\def\blem{\begin{lema}}
%\def\elem{\end{lema}}
%\def\bcor{\begin{cor}}
%\def\ecor{\end{cor}}
%\def\bp{\noindent{\it Proof. }}
%\def\ep{\noindent{\hfill $\Box$}}
%%%%% END USER-DEFINED MACROS HERE %%%%%

\title[Directional Recurrence]{Directional Recurrence for Infinite Measure Preserving $\mathbb Z^d$ actions}
%\date{}                                           % Activate to display a given date or no date
%authorinfo 1
\author{Aimee S. A. Johnson}
\address[Johnson]{Department of Mathematics and Statistics, Swarthmore College, Swarthmore, PA 19081}
 
%authorinfo 3
\author{Ay\c se A. \c Sah\.in }
\address[\c Sahin]{Department of Mathematical Sciences, DePaul University, 2320 N. Kenmore Ave., Chicago, IL 60626}

\dedicatory{Dedicated to the memory of Daniel J. Rudolph}
\date{May 8, 2014}

%\begin{document}
\maketitle
%\section{}
%\subsection{}

\begin{abstract}
We define directional recurrence for infinite measure preserving $\mathbb Z^d$ actions both intrinsically and via the unit suspension flow and prove that the two definitions are equivalent. We study the structure of the set of recurrent directions and
show it is always a $G_{\delta}$ set.  We construct an example of a recurrent action with no recurrent directions, answering a question posed in a 2007 paper of Daniel J. Rudolph.  We also show by example that it is possible for a recurrent action to not be recurrent in an irrational direction
even if all its sub-actions are recurrent.  
\end{abstract}

\section{Introduction}
Given a dynamical system defined by the action of a group $G$, it is natural to study the sub-dynamics of the action.  In particular, one can ask what dynamical properties of the action of $G$ are inherited by the dynamical systems one obtains by restricting the original action to sub-groups of $G$.  In the 1980's Milnor introduced the more general idea of {\it directional dynamics} \cite{Mil}.  He defined the directional entropy of a $\mathbb Z^d$ action in all (i.e. including irrational) directions.  The study of directional entropy has been a productive  line of research  (see for example \cite{Si},  \cite{POsaka}, \cite{PIsrael}, \cite{RSdirent}).  In addition, the idea of defining directional dynamical properties more generally has led to other advances in dynamics, most notably expansive sub-dynamics introduced in \cite{BL}.  In this paper we define directional recurrence for infinite measure preserving $\mathbb Z^d$ actions and we study the structure of the set of recurrent directions under the assumption that the action is also recurrent. 

The motivation for the project originally was a question posed by Rudolph \cite{Rudfol} in response to Feldman's proof in \cite{Fel} of the ratio ergodic theorem for conservative $\mathbb Z^d$ actions.  Feldman's proof required that  the generators of the action also be recurrent.  Rudolph asked in \cite{Rudfol} whether this was an additional assumption or if it is the case that the recurrence properties of a group action are necessarily inherited by its sub-actions.  In 2008 the authors answered the question in the negative.  

\begin{theorem}\label{dnorat}
There exists an infinite measure preserving, recurrent, and ergodic $\mathbb Z^2$ action, $(X,\mu,\{T^{\vec n}\}_{\vec n\in\mathbb Z^2})$, on a $\sigma$-finite Lebesgue space, with the property that $(X,\mu,\{T^{k\vec n}\}_{k\in\mathbb Z})$ is not recurrent for all $\vec n\in\mathbb Z^2$.
\end{theorem}

The question became moot from the point of view of the ratio ergodic theorem when Hochman \cite{Hochratio} gave an alternative proof that does not require the recurrence of any sub-actions.  When Rudolph saw the example constructed to prove Theorem~\ref{dnorat}, and others displaying a range of possibilities for the recurrence properties of particular sets of sub-actions, he suggested that we extend the definition of directional recurrence to include irrational directions and that we analyze the structure of possible recurrent directions in this more general setting.  The results on directional recurrence presented here were established jointly with Rudolph, but unfortunately the final exposition was completed after his untimely death in 2010.   The authors are grateful to have had this chance to collaborate with Rudolph who was both their friend and advisor.

For ease of exposition from here on out we restrict our attention to $d=2$ but we note that the proofs and examples generalize readily to all $d>1$.  

Denote the directions in $\mathbb R^2$ by angles in $[0,\pi)$.  There are two ways of defining directional properties for $\mathbb Z^2$ actions.       Milnor's approach is to define the property intrinsically in the $\mathbb Z^2$ action.  Namely, we say that an action $\{T^{\vec n}\}_{\vec n\in\mathbb Z^2}$ has a dynamic property in the direction $\theta$ if there exist a collection of arbitrarily good rational approximants $\vec n_i$ of $\theta$ so that the set  $\{T^{\vec n_i}\}$ has that property.  Alternatively, given a direction $\theta$, one can associate to $\{T^{\vec n}\}_{\vec n\in\mathbb Z^2}$ an $\mathbb R$ action in that direction by considering the unit suspension flow of $\{T^{\vec n}\}_{\vec n\in\mathbb Z^2}$, restricted to the direction $\theta$.  For directional entropy the two approaches yield equivalent definitions \cite{PIsrael}.  Here we define directional recurrence both intrinsically and via the unit suspension and we prove that the two definitions are equivalent.

\begin{theorem}\label{same}
An ergodic, infinite measure preserving $\mathbb Z^2$ action on a $\sigma$-finite Lebesgue space is recurrent in a direction $\theta$ if and only if its unit suspension restricted to the direction $\theta$ is recurrent as an $\mathbb R$ action.
\end{theorem}

 We show that there are some restrictions on the types of subsets of $[0,\pi)$  that can appear as recurrent directions for a $\mathbb Z^2$ action.  Deferring the formal definition until later, let $\mathcal R_T\subset[0,\pi)$ denote the directions of recurrence for a $\mathbb Z^2$ action $\{T^{\vec n}\}_{\vec n\in\mathbb Z^2}$.

\begin{theorem}\label{Gdelta}  The set of recurrent directions, $\mathcal R_T$, of an 
infinite measure preserving, recurrent, and ergodic $\mathbb Z^2$ action on a $\sigma$-finite Lebesgue space, $(X,\mu,\{T^{\vec n}\}_{\vec n\in\mathbb Z^2})$ is a $G_\delta$ subset of $[0,\pi)$.
\end{theorem}

There are directional properties with stronger restrictions on the structure of the bad set of directions.  For example an ergodic and probability measure preserving $\mathbb Z^2$ action can have at most countably many directions along which it is not ergodic.  A similar statement holds for directional weak mixing for a weak mixing $\mathbb Z^2$ action (see for example \cite{RRS}).  
Here we show that there is a great deal more flexibility in the structure of $\mathcal R_T$.  In particular,  the original example proving Theorem~\ref{dnorat} has no recurrent directions even under the extended definition of directional recurrence.   

\begin{theorem}\label{none}
There exists an infinite measure preserving, recurrent, and ergodic $\mathbb Z^2$ action  on a $\sigma$-finite Lebesgue space, $(X,\mu,\{T^{\vec n}\}_{\vec n\in\mathbb Z^2})$, with the property that $\mathcal R_T=\emptyset$.
\end{theorem}

The example proving Theorems~\ref{dnorat} and~\ref{none} is a rank one action that is constructed using a cutting and stacking procedure.   The construction has enough flexibility built into it so that many specific examples can be obtained by choosing the parameters of the procedure appropriately.  Note that by Theorem~\ref{Gdelta} we know that $\mathcal R_T$ cannot consist only of the rational directions in $[0,\pi)$.  Using our technique we construct an explicit example that shows that it is possible for $\mathcal R_T$ to contain all the rational directions, but still not be all of $[0,\pi)$.  
  
\begin{theorem}\label{all}
Let $\alpha_1,\ldots,\alpha_k\in[0,\pi)$ be irrational.  There exists an infinite measure preserving, recurrent, and ergodic $\mathbb Z^2$ action on a $\sigma$-finite Lebesgue space, $(X,\mu,\{T^{\vec n}\}_{\vec n\in\mathbb Z^2})$,  with the property that $\mathcal R_T$ contains all rational directions but $\alpha_i\notin\mathcal R_T$ for $i=1,\ldots,k$.
\end{theorem}

It is, however, an open question as to whether any $G_\delta$ subset of $[0,\pi)$ can be achieved as a set of recurrence.
 
The organization of the paper is as follows.  In Section~\ref{definitions}  we define directional recurrence intrinsically and via the unit suspension flow and prove that the two definitions are equivalent.   We also show that for rational directions the definition coincides with the usual definition of recurrence for the $\mathbb Z$ action obtained by restricting to the corresponding subgroup of $\mathbb Z^2$.  In Section~\ref{structure} we give a third characterization of recurrence in terms of a {\it sweeping out property} and we prove Theorem~\ref{Gdelta}. In Section~\ref{exs}  we establish the notation we use to construct the cutting and stacking examples,  give some basic properties of infinite measure preserving rank one actions and  the proofs of Theorems~\ref{none} and \ref{all}.

\section{Directional recurrence}\label{definitions}
In what follows we let $X$ denote a $\sigma$-finite Lebesgue space and $(X,\mu,\{T^{\vec n}\}_{\vec n\in\mathbb Z^2})$ denote an infinite measure preserving, and ergodic $\mathbb Z^2$ action on $X$.   Our goal in this section is to extend the notion of recurrence to include directions that do not correspond to a subgroup action of $\{T^{\vec n}\}_{\vec n\in\mathbb Z^2}$.  We begin by establishing some notation to describe directions and vectors associated to a direction.  We then define directional recurrence intrinsically and via the unit suspension.  Finally, we prove that the definitions are equivalent.

\subsection{Notation}
Recall that a direction is an angle $\theta\in[0,\pi)$.  For ease of exposition we introduce some terminology to help associate directions to vectors in the plane.  

\begin{notation}
Let $\vec u=(u_1,u_2)\in\mathbb R^2$.   If $u_2\neq 0$ then the {\it direction associated to $\vec u$} is the angle it makes with the vector $\text{\em sign}(u_2)\cdot\displaystyle{(1,0)}$.  If $u_2=0$, then the direction associated to $\vec u$ is $0$.
\end{notation}

\begin{notation}
Given $\theta\in[0,\pi)$, a vector $\vec v=\vec v(\theta)\in\mathbb R^2$  is said to be {\em associated to the direction $\theta$} if the direction associated to $\vec v$ is $\theta$.
\end{notation}

\begin{notation}
A direction $\theta\in[0,\pi)$ is rational if it has an associated vector $\vec n\in\mathbb Z^2$.
\end{notation}  

The concept of a tunnel is the key geometric tool that we will use to define directional recurrence.

\begin{definition}\label{tunnel} Let $\theta\in[0,\pi)$ and $\epsilon>0$.  The $\epsilon$-tunnel of  $\theta$  is the set of points in $\mathbb R^2$ within $\epsilon$ of the line through the origin in direction $\theta$, namely the set
      \begin{equation*}
      \bigcup_{t\in\mathbb R}B_{\epsilon}(t\vec u)
      \end{equation*}
where $B_{\epsilon}(\vec x)$ denotes the $\epsilon$-ball around the point $\vec x\in\mathbb R^2$, and $\vec u\in\mathbb R^2$ is any vector associated to $\theta$.
 \end{definition}

Finally, we introduce some notation which we will use in working with unit suspension flows.
\begin{notation}  Recall $\lfloor x \rfloor$ is the greatest integer $n\le x$.  For a vector $\vec{v} = (v_1, v_2) \, \in\mathbb{R}^2$, we set $\lfloor \vec{v} \rfloor=(\lfloor v_1 \rfloor, \lfloor v_2 \rfloor ) \in \mathbb{Z}^2$.  For a set $R\subset \mathbb{R}^2$, $\lfloor R \rfloor=\{\lfloor \vec{v} \rfloor: \vec{v}\in R\}$.

Recall also that $\{x\}$ is the fractional part of the real number $x$, i.e. $\{x\} = x - \lfloor x \rfloor$.  We define $\{ \vec{v} \} = ( \{v_1\}, \{v_2 \})$ and $\{ R\}=\{\{\vec v\}:\vec v\in R\}$.
\end{notation}

\subsection{Defining directional recurrence by rational approximants}

We begin by recalling the definition of a recurrent group action.  For an extensive treatment of the properties of recurrent actions see for example \cite{A}.

\begin{definition}\label{usualrecurrencedef} An ergodic, infinite measure preserving action of a group $G$ on a $\sigma$-finite Lebesgue space,
  $(X,\mu,\{T^{\gamma}\}_{\gamma\in G})$, is recurrent if for every measurable set $A$ with $\mu(A)>0$, there exists $\gamma\in G$ such that $\mu(A\cap T^{\gamma} A)>0$.
\end{definition}

We are now ready to define directional recurrence for actions of the group $\mathbb Z^2$.

\begin{definition}\label{recurrentdirection}
A direction $\theta$ is recurrent for an infinite measure preserving, ergodic, and recurrent $\mathbb Z^2$ action on a $\sigma$-finite Lebesgue space, $(X,\mu,\{T^{\vec n}\}_{\vec n\in \mathbb Z^2})$, if for all 
    $\epsilon > 0$ and all measurable sets $A$ of positive measure 
    there is a vector $\vec{n}$ in the $\epsilon$-tunnel of 
    $\theta$ so that $\mu (A \cap T^{\vec{n}}A ) > 0$.

    \end{definition}

\begin{notation}
We denote the set of recurrent directions for $(X,\mu,\{T^{\vec n}\}_{\vec n\in \mathbb Z^2})$ by $\mathcal R_T$.  
\end{notation}

The next result shows that recurrence along a rational direction as given by Definition~\ref{recurrentdirection} is equivalent to the recurrence of the group elements corresponding to that direction as given by Definition~\ref{usualrecurrencedef}.

\begin{lemma} 
Let $(X,\mu,\{T^{\vec n}\}_{\vec n\in \mathbb Z^2})$ be an infinite measure preserving $\mathbb Z^d$ action on a $\sigma$-finite Lebesgue space.
A rational direction $\theta\in[0,\pi)\in\mathcal R_T$ if and only if $(X,\mu,\{T^{k\vec n}\}_{k\in\mathbb Z})$ is a recurrent $\mathbb Z$ action for all $\vec n\in\mathbb Z^2$ whose associated direction is $\theta$.
    \end{lemma}
\begin{proof}
Fix $(X,\mu,\{T^{\vec n}\}_{\vec n\in\mathbb Z^2})$ as in the statement of the lemma and suppose there is a rational $\theta\in\mathcal R_T$.  Let $\vec n\in\mathbb Z^2$ be a vector with associated direction $\theta$ and choose $\epsilon>0$ that is smaller than the distance between the line $\{t\vec n\}_{t\in\mathbb R}$ and those elements of $\mathbb Z^2$ that do not lie on the line $\{t\vec n\}_{t\in\mathbb R}$.  For any set $A$ of positive measure the vector $\vec w\in\mathbb Z^2$ satisfying Definition~\ref{recurrentdirection} for $\theta$ with this $\epsilon$ must necessarily be of the form $k\vec n$, showing that $T^{\vec n}$ satisfies Definition~\ref{usualrecurrencedef}.  The converse is immediate.
\end{proof}

\subsection{Directional recurrence via the unit suspension}  We begin by recalling the definition of the unit suspension of the $\mathbb Z^2$ action $(X,\mu,\{T^{\vec v}\}_{\vec v\in\mathbb Z^2})$.

\begin{definition}  The unit suspension of  a $\mathbb{Z}^2$ action on a $\sigma$-finite measure space, $(X,\mu,\{T^{\vec n}\}_{\vec n\in\mathbb Z^2})$, is the $\mathbb{R}^2$-action $\hat{T}$ defined on $X\times [0,1)^2$ by
$$\hat{T}^{\vec{v}}(x,\vec{r}) = (\, T^{(\lfloor \vec{v}+\vec{r} \rfloor} x, \{ \vec{v}+\vec{r} \} \,). $$
Note that $\hat T$ preserves the product measure $\mu\times\lambda$ where $\lambda$ is Lebesgue measure on $[0,1)^2$.
\end{definition}

We now prove Theorem~\ref{same}, restated more precisely below.
\begin{theorem}
Let $(X,\mu,\{T^{\vec n}\}_{\vec n\in\mathbb Z^2})$ be an infinite measure preserving $\mathbb Z^2$ action on a $\sigma$-finite Lebesgue space and let $(X\times[0,1)^2,\mu\times\lambda,\{\hat T^{\vec v}\}_{\vec v\in\mathbb R^2})$ denote its unit suspension.  
 The direction $\theta\in\mathcal R_T$ 
 if and only if the $\mathbb R$ action $(X \times[0,1)^2,\mu\times\lambda,\{\hat T^{t\vec u}\}_{t\in\mathbb R})$ is recurrent for any $\vec u$ associated to the direction $\theta$.
\end{theorem}

\begin{proof}
Assume first that the $\mathbb{Z}^2$ action is recurrent in direction $\theta$.  Because the proofs of Propositions~\ref{otherway} and \ref{unifsweepgivesrecur} hold verbatim for $\mathbb R^2$ actions in order to prove the unit suspension is recurrent in direction $\theta$ it suffices to show that it is recurrent for sets of the form  $A\times R$ where $A\subset X$ is a subset of positive measure and $R\subset[0,1)^2$ is a rectangle $[i_1, i_2]\times [j_1, j_2]$.    Fix $\epsilon>0$ to be less than $\frac{1}{4}\min\{i_2- i_1, j_2-j_1 \}$.
By the directional recurrence of the $\mathbb Z^2$ action we can find $\vec n\in\mathbb Z^2$ in the $\epsilon$-tunnel of $\theta$ so that $\mu(A\cap T^{\vec{n}}A)>0$.

Choose $\vec u\in\mathbb R^2$ to be the vector associated with direction $\theta$ that is closest to $\vec n$.  Note that $\Vert\vec u-\vec n\Vert<\epsilon$.  
Let $\tilde R\subset R$ be the subset of $R$ consisting of those $(i,j)\in R$ that are at least a fraction $\frac{1}{4}\min\{ i_2- i_1, j_2-j_1 \}$ away from the edges of the square.  
By our choice of $\epsilon$  we then know that if $(s,t)\in\tilde R$ then $\vec u+(s,t)\in \vec{n}+R$ and therefore $\lfloor  \vec{u}+ \tilde{R}  \rfloor = \vec{n}$ and $\{\vec u+\tilde R\}\subset   R$.  Thus, if we set $\tilde A=A\cap T^{\vec n}A$ we have $\hat T^{\vec u}(\tilde A\times \tilde R)=(T^{\vec n}\tilde A\times R')$ for some $R'\subset R$ with area at least half of the area of $R$ and therefore
$\mu\times\lambda\left(\hat T^{\vec u}(\tilde A\times \tilde R)\cap(A\times R)\right)\ge\frac12 \, \mu\times\lambda(\tilde A\times\tilde R)>0$.
 
Now suppose that the unit suspension is recurrent in a direction $\theta$. 
 Fix $A\subset X$, a set of positive measure and $\epsilon>0$.  
Let $R\in[0,1)^2$ be a square based at the origin of side length $ \frac{1}{4}\epsilon$.  
By the recurrence of the unit suspension there is a vector $\vec{u}$ associated to $\theta$ for which the measure of $(A\times R)\cap \hat{T}^{\vec{u}}(A\times R)$ is positive.  
Consider $(x,\vec{r})$ in this intersection.   Since $\hat T^{\vec{u}}(x,\vec r)=(T^{\lfloor\vec u+\vec r\rfloor}x,\{\vec u+\vec r\})\in A\times R$ we have $\vec n=\lfloor\vec u+\vec r\rfloor$ with $T^{\vec n}x\in A$.  It remains to prove that $\vec n$ must lie in the $\epsilon$-tunnel of $\theta$.  To see this note that both $\vec r$ and $\vec j=\{\vec u+\vec r\}$ are in $R$.   On the other hand $\vec j=\vec u+\vec r-\vec n$ so $\Vert\vec u-\vec n\Vert=\Vert\vec j-\vec r\Vert$ and the latter is less than $\epsilon$ by the definition of $R$.  
 
\end{proof}

\section{The structure of $\mathcal R_T$}\label{structure}
\subsection{Uniform sweeping out}
We give an alternative characterization of recurrence that is useful for our arguments in this section.  In what follows we fix an infinite measure preserving $\mathbb Z^2$ action on a $\sigma$-finite Lebesgue space $(X,\mu,\{T^{\vec n}\}_{\vec n\in\mathbb Z^2})$ and we denote the $\sigma$-algebra of $\mu$-measurable sets by $\mathcal B$.

\begin{definition} \label{unifsweep}
    Let $A\in\mathcal{B}$ be such that $0<\mu(A)<\infty$.  A 
    direction $\theta$ is said to have the $\mathbf{\epsilon}${\bf{  
    sweeping out property}} for $A$ if for all $\alpha$ with $0<\alpha < 
    \frac{1}{2}$, there exist pairwise disjoint subsets $A_{1}, A_{2},\ldots,A_{k}$ 
    of $A$ of positive measure and vectors 
    $\vec{v}_{1},\ldots,\vec{v}_{k}$ in the $\epsilon$-tunnel of 
    $\theta$ so that the sets $T^{\vec v_1}A_1, \ldots, T^{\vec v_k}A_k$ are a pairwise disjoint collection of subsets of $A$ satisfying:
    
    \begin{gather}
     \mu(\bigcup_{i=1}^{k}A_{i}) > \alpha\cdot\mu(A),\text{ and}\label{AcoverA}\\
     \mu(\bigcup_{i=1}^{k}T^{\vec{v}_{i}}(A_{i})) > \alpha\cdot\mu(A)\label{TAcoverA}.
    \end{gather}

 If $\theta$ has the $\epsilon$ sweeping out property for $A$ for every $\epsilon>0$  we say it is {\bf{uniform sweeping 
    out}} for $A$.
\end{definition}

The next two results show that the recurrence of a direction can be characterized in terms of Definition~\ref{unifsweep}.
\begin{proposition}\label{otherway}
    A recurrent direction is uniform sweeping out for all measurable 
    sets $A$ of finite positive measure.
\end{proposition}
\begin{proof}
    Let $\theta$ be a recurrent direction.  Choose $\epsilon>0$, 
    $0<\alpha < \frac{1}{2}$, and $A$, a measurable set of finite 
    positive measure.  Since $\theta$ is recurrent, there is a vector 
    $\vec{v}_{1}$ in the $\epsilon$-tunnel of $\theta$ such 
    that $\mu(A\cap T^{\vec{v}_{1}}A)>0$.  Define
  $\hat{A}_{1} = A\cap T^{\vec{v}_{1}}A$ and $A_{1}=T^{-\vec{v}_{1}}\hat{A}_{1}.$

    Assume we have found $\vec{v}_{j}$, $\hat{A}_{j}$, and $A_{j}$ 
    for $1\le j\le i$ and set $A_{i+1}^{'}= A - 
    (\bigcup_{j=1}^{i}(\hat{A}_{j} \cup A_{j})\,)$.  If 
    $\mu(A_{i+1}^{'})>0$, use the recurrence of $\theta$ to 
    find $\vec{v}_{i+1}$ in the $\epsilon$-tunnel of $\theta $ so that
    $\mu(A_{i+1}^{'} \cap T^{\vec{v}_{i+1}}A_{i+1}^{'}) > 0.$
    Define
    \begin{equation*}
   \hat{A}_{i+1} = A_{i+1}^{'}\cap T^{\vec{v}_{i+1}}A_{i+1}^{'}\qquad\text{and}\qquad 
    A_{i+1}=T^{-\vec{v}_{i+1}}\hat{A}_{i+1}.
    \end{equation*}
    Fix $\vec{v}_{1}$ and note that the choices of $\vec v_i$ with $i>1$ described above are not unique.   Consider all possible extensions of 
    $(A_{1}, T^{\vec{v}_{1}}A_{1})$ as defined above:
    \begin{equation*}
    (A_{1}, T^{\vec{v}_{1}}A_{1}), \bigg( \bigcup_{i=1}^{2}A_{i}, \, \bigcup_{i=1}^{2} 
    T^{\vec{v}_{i}}A_{i}\bigg),\cdots,
\bigg( \bigcup_{i=1}^{k}A_{i}, \, \bigcup_{i=1}^{k} 
    T^{\vec{v}_{i}}A_{i}\bigg),\cdots .
    \end{equation*} 
 
Each such sequence forms a chain in the partially 
ordered set of pairs of measurable subsets of $A\times A$ under 
    the relation $(B,C)\le (B', C')$ if and only if $B\subset B'$ and 
    $C\subset C'$.  Since $(A,A)$ is a maximal element of this set, 
    each chain is contained in a maximal chain.  Choose such a 
    maximal chain and note that if 
    $$\mu \big(\bigcup_{i=1}^{\infty}(A_{i}\cup T^{\vec{v}_{i}}A_{i}) 
    \big) < 
    \mu(A)$$
    it would be possible to extend the chain; therefore we must have 
    equality.  But then it follows that
    \begin{align*}
       \mu(A) = \mu(\bigcup_{i=1}^{\infty}(A_{i}\cup 
    T^{\vec{v}_{i}}A_{i}) ) &\le \sum_{i=1}^{\infty}\mu(A_{i}\cup 
    T^{\vec{v}_{i}}A_{i})\\
   &\le \sum_{i=1}^{\infty} 2\, \mu(A_{i})=2\sum_{i=1}^{\infty}\mu(A_{i}).
    \end{align*}
    We can thus find $k$ so that 
    $$\mu(\bigcup_{i=1}^{k}A_{i}) = \sum_{i=1}^{k} \mu(A_{i}) > \alpha\, \mu(A),$$
    satisfying (1).
    
    Since the action is measure preserving 
    \begin{equation*}
    \sum_{i=1}^{k} \mu(A_{i})=\sum_{i=1}^{k}\mu(T^{\vec{v}_i}A_{i})
    \end{equation*}
    and (2) is also satisfied.
     \end{proof}

The following is a converse result, showing that if a direction has the uniform sweeping out property for a sufficiently good collection of sets, then it is a recurrent direction. 

\begin{proposition}\label{unifsweepgivesrecur}
    Let $\mathcal{A}$ be a countable collection of sets of finite 
    positive measure which are dense in the $\sigma$-algebra 
    $\mathcal{B}$. If $\theta$ is a direction which is uniform 
    sweeping out for all $A\in\mathcal{A}$, then $\theta$ is a 
    recurrent direction.
\end{proposition}

\begin{proof}
Let $C$ be a subset of $X$ with infinite measure.  Since $X$ is $\sigma$-finite we can write $C=\cup _{i=1}^\infty C_i$ with $C_i\subset C_{i+1}$ and $0<\mu (C_i) <\infty$ for all $i$.  Thus it suffices to prove that $\theta$ is a direction of recurrence for sets of finite positive measure.

Fix a set $C$ of finite, positive measure and $\epsilon>0$.  As the collection $\mathcal A$ is dense there is a set $A\in\mathcal A$ of positive, finite measure, with the property that 
\begin{equation}\label{AandC}
\mu(A\triangle C)<\epsilon \, \mu(A).
\end{equation}

Fix $\frac14<\alpha<\frac12$.  Choose sets $A_i\subset A$ and vectors $\{\vec v_i\}$ in the $\epsilon$-tunnel of $\theta$, $i=1,\ldots,k$, satisfying Definition~\ref{unifsweep} for the direction $\theta$.  The rest of the argument, while technical, relies on the idea that since the sets $\cup A_i$ and $\cup T^{\vec v_i} A_i$ both cover at least 
$\alpha$ of the set $A$, and $A$ is an approximation of the set $C$,  provided $\epsilon$ is chosen small enough, there is an index $i$ with the property that both
\begin{equation}\label{want}
\mu(A_i\cap C)>\frac12\mu(A_i)\qquad\text{and}\qquad\mu(T^{\vec v_i}A_i\cap C)>\frac12\mu(T^{\vec v_i} A_i).
\end{equation}
We thus have $\vec v_i\in\mathbb Z^2$ in the $\epsilon$-tunnel of $\theta$ for which $\mu(C\cap T^{\vec v_i}C)>0$. Now the details.

It follows easily from \eqref{AandC} that if we consider $\cup A_i \cap C$, rather than $\cup A_i$, we lose a set of small measure.  More formally
\begin{align*}
\mu\left(\cup_{i=1}^kA_i \cap C\right)&\ge\mu\left(\cup_{i=1}^kA_i\right)-\mu(A\triangle C)\\
&\ge \mu\left(\cup_{i=1}^kA_i\right)-\epsilon\, \mu (A).
\end{align*}
On the other hand, our choice of $\alpha$ and the statement in  \eqref{AcoverA} gives that $4\mu(\cup_{i=1}^k A_i)>\mu(A)$ so we have
\begin{equation*}
\mu\left(\cup_{i=1}^kA_i \cap C\right)>(1-4\epsilon)\mu(\cup_{i=1}^kA_i).
\end{equation*}

We now argue that the collection of sets $A_i$ that are well covered by $C$ make up most of the union, in measure.  Let
 $\mathcal{I} \subset \{ 1,..., k\}$ be the collection of $i$'s such that $\mu(A_i \cap C)> (1-2\sqrt{\epsilon}) \mu(A_i)$.

Note that
\begin{align*}
  \mu\left(\cup_{i=1}^kA_i \cap C \right) &= \mu\left(\cup_{i\in\mathcal{I}} A_i \cap C \right)  + \mu\left(\cup_{i\notin\mathcal{I}} A_i \cap C \right) \\
& \le \mu\left(\cup_{i\in\mathcal{I}} A_i  \right) + (1-2\sqrt{\epsilon})\mu\left(\cup_{i\notin\mathcal{I}} A_i  \right) \\
& = \mu\left(\cup_{i=1}^k A_i  \right)  - 2\sqrt{\epsilon}\,  \mu\left(\cup_{i\notin\mathcal{I}} A_i  \right).
\end{align*}
We thus have 
\begin{equation*}
(1-4\epsilon)\mu\left(\cup_{i=1}^k A_i\right)<\mu\left(\cup_{i=1}^kA_i\cap C\right)\le\mu\left(\cup_{i=1}^k A_i\right)-2\sqrt\epsilon\mu\left(\cup_{i\notin\mathcal{I}} A_i  \right)
\end{equation*}
which implies that 
\begin{equation*}
 \mu\left(\cup_{i\notin\mathcal{I}} A_i  \right) \le  2\sqrt{\epsilon} \mu\left(\cup_{i=1}^kA_i \right).
 \end{equation*}
 
We now make an analogous argument for the sets $T^{\vec v_i}A_i$.  We begin by using \eqref{TAcoverA} to argue, similarly to before, that we have
\begin{equation*}
\mu\left(\cup_{i=1}^kT^{\vec v_i}A_i \cap C\right)>(1-4\epsilon)\mu\left(\cup_{i=1}^kT^{\vec v_i}A_i\right).
\end{equation*}
We can similarly define $\mathcal{J} \subset \{ 1,..., k\}$ to be the collection of $j$'s such that $\mu(T^{\vec v_j}A_j \cap C)> (1-2\sqrt{\epsilon}) \mu(A_j)$ and show that 
$$2\sqrt{\epsilon} \mu\left(\cup_{j=1}^k T^{\vec v_j}A_j \right) =2\sqrt\epsilon\mu\left(\cup_{j=1}^k A_j\right)\ge  \mu\left(\cup_{j\notin\mathcal{J}} A_j  \right)=\mu\left(\cup_{j\notin\mathcal J} T^{\vec v_j}A_j\right).$$

Finally, if we consider only those indices that lie in both $\mathcal I$ and $\mathcal J$ we can conclude
\begin{align*}
\mu\left(\cup_{i\in\mathcal{I}\cap\mathcal{J}} A_i  \right) & \ge (1-4\sqrt{\epsilon} ) \mu\left(\cup_{i=1}^kA_i \right) \mbox{ and} \\
\mu\left(\cup_{i\in\mathcal{I}\cap\mathcal{J}} T^{\vec v_i}A_i  \right) & \ge (1-4\sqrt{\epsilon} ) \mu\left(\cup_{i=1}^k T^{\vec v_i}A_i \right).
\end{align*}
Thus $\mathcal I\cap\mathcal J\neq\emptyset$ and for any $i\in \mathcal{I}\cap \mathcal{J}$  we will have 
\begin{equation*}
\mu(A_i\cap C)>(1-2\sqrt\epsilon)\mu(A_i) \mbox{ and}
\end{equation*}
\begin{equation*}
\mu(T^{\vec v_i}A_i\cap C)>(1-2\sqrt\epsilon)\mu(T^{\vec v_i}A_i).
\end{equation*}
Clearly if $\epsilon$ is chosen small enough then  \eqref{want} holds.

\end{proof}
\subsection{Proof of Theorem~\ref{Gdelta}}
We now have all the ingredients in place to prove that $\mathcal R_T$ must be a $G_\delta$ set.

Let $\mathcal{A}=\{ A_i \}$ be a countable collection of sets of 
    finite positive measure which are dense in $\mathcal{B}$.  Then 
    by Propositions~\ref{otherway} and~\ref{unifsweepgivesrecur}, we have 
    $$\mathcal{R}_T = \cap_{i=1}^{\infty} \{ \theta : \theta \mbox{ 
    has the uniform sweeping out property under $T$ for } A_i \}.$$ 
    By Definition \ref{unifsweep}, we can rewrite this set as:
    $$\mathcal{R}_T = \cap_{i=1}^{\infty} \, \cap_{n=1}^{\infty} \{ 
    \theta: \mbox{ $\theta$ has the $\frac{1}{n}$ sweeping out 
    property under $T$ for } A_i \}.$$  
    Thus we just need to show that for any set $A\in\mathcal B$ of positive finite measure, a set of the form
     $$\{ \theta : \theta \mbox{ has the $\epsilon$ sweeping out property 
    under $T$ for set $A$} \}$$ is open.  

Fix such a set $A$ and $\epsilon >0$, and let $\theta$ be in the set 
defined above. Fix a vector $\vec{v}_i$ in the $\epsilon$-tunnel 
of $\theta$ given by Definition \ref{unifsweep} and denote its associated direction by $\omega_i$.  By the continuity of the sine function there exists an $\eta_i = \eta(\vec v_i)$ so that $\vec v_i$ lies in the $\epsilon$-tunnel of all directions in the interval $(\omega_i-\eta_i,\omega_i+\eta_i)$.  Further, $\eta_i$ can be chosen in such a way that $\theta$ lies in this interval.  Choose $\delta_i$  so that $(\theta-\delta_i,\theta+\delta_i)\subset(\omega_i-\eta_i,\omega_i+\eta_i)$.  Let $\delta=\min_{i=1,\cdots,k} \delta_i$.  
Then the 
interval $(\theta - \delta, \theta + \delta)$ consists of directions 
that contain all the $\vec{v}_i$'s in their $\epsilon$-tunnel, and thus all 
have the $\epsilon$-uniform sweeping out property under $T$ for the set $A$, as wanted.

\section{Examples of $\mathcal R_T$}\label{exs}
\subsection{Infinite measure preserving rank one actions}
The examples we construct in this paper will be rank one infinite measure preserving
actions defined on measurable subsets of $\mathbb R^+$ .  We will
construct them using standard cutting and stacking methods.
In this section we establish the notation and facts about such
actions which we need for our arguments.  For more general
discussions of rank one transformations and cutting and stacking we refer the
reader to \cite{Fersurvey}, for rank one non-singular transformations to \cite{RudSilva}, and for discussions of rank one actions in higher dimensions to \cite{JS1} and \cite{RSdirent}.

The following notation will be of use to us in our constructions.
\begin{notation}
Let $B\subset\mathbb R^2, \vec v\in\mathbb R^2$.  The $\vec v$-tunnel of $B$ is the set   $B+\{t\vec v\}_{t\in\mathbb R}$.
\end{notation}

 \begin{notation}
For $n\in\mathbb N$ we set $B_n=[0,n)^2\cap\mathbb Z^2$ and $\overline B_n=(-n,n)^2\cap\mathbb Z^2$.  Given a square $B=\vec v+B_n$ or $\vec v +
\overline{B}_{n}$ we call the
point $\vec v$ the {\it base point} of $B$.  A vector $\vec
v=(x,y)\in\mathbb Z^2$ with $x,y>0$ will be called a {\it positive
vector}.
\end{notation}

Fix $n_1\in\mathbb N$ and let $\{I^1_{\vec v}\}_{\vec v\in B_{n_1}}$ be a pairwise disjoint collection of equal finite length interval subsets of $\mathbb R$ indexed by $B_{n_1}$.   Geometrically we think of the intervals as arranged in the shape of $B_{n_1}$ according to their indexing vector.  We start defining an action $T$ by mapping levels to each other by translation:
\begin{equation*} 
T^{\vec w}(I^1_{\vec v})=I_{\vec v+\vec w}^1
\end{equation*}
whenever $\vec v,\vec w\in B_{n_1}$ are such that $\vec v+\vec w\in
B_{n_1}$.  
The collection of intervals, denoted by $\tau_1$, is clearly a tower for the $\mathbb Z^2$ action.  Each interval will be called a {\it level} of the tower and $I_{\vec 0}^1$ will be called the {\it base} of the tower.  

Given a tower $\tau_i, i\ge 1$, we define the next tower 
$\tau_{i+1}$ by
dividing the base $I^i_{\vec 0}$ into $k(i)$ equal length subintervals, $I^{i,j}_{\vec 0}, j=1,\ldots,k(i)$ for some choice of $k(i)\in\mathbb N$.   For
$j=1,\cdots,k(i)$ we call the subsets 
\begin{equation*}
\tau_i^j=\cup_{\vec v\in B_{n_i}} T^{\vec v}I^{i,j}_{\vec 0}
\end{equation*}
the {\it slices} of $\tau_i$. 
The tower $\tau_{i+1}$ is then constructed by first placing these slices of $\tau_i$ at $k(i)$ locations in $\mathbb Z^2$ so that they do not intersect.  These locations are denoted by $\mathcal P_i(1)$.  In our constructions, $\mathcal P_i(1)$ will always include $\vec 0$.  Define $n_{i+1}\in\mathbb N$ to be the smallest integer so that $B_{n_{i+1}}$ contains the set
\begin{equation*}
\bigcup_{\vec v\in\mathcal P_i(1)} ( B_{n_i}+\vec v )
\end{equation*} 
and set $\mathcal S_{i+1}=B_{n_{i+1}}\setminus\big(\cup_{\vec
v\in\mathcal P_i} B_{n_i}+\vec v)$.  For each $\vec v\in\mathcal
S_{i+1}$ we choose an interval $s_{\vec v}$ of length equal to that of the intervals constituting the slices of $\tau_i$ in such a way that the resulting collection of intervals is pairwise disjoint.  These are called the {\it spacer intervals} of the $i+1$ stage, or of $\tau_{i+1}$.  
To complete the construction of $\tau_{i+1}$ we assign spacer intervals to the locations in $B_{n_{i+1}}$ not occupied by the slices of $\tau_i$ according to their indexing vector. 

We rename the intervals $\{ I_{\vec{v}}^{i+1} \}_{\vec{v}\in 
B_{n_{i+1}}}$ with $I_{\vec 0}^{i+1}$ denoting  the base of the tower.  Finally, the definition of $T$ is extended to $\tau_{i+1}$ by setting
 \begin{equation*} 
T^{\vec w}(I^{i+1}_{\vec v})=I_{\vec v+\vec w}^{i+1}
\end{equation*}
whenever $\vec v,\vec w\in B_{n_{i+1}}$ are such that $\vec v+\vec
w\in
B_{n_{i+1}}$.  

Let $X$ denote the subset of $\mathbb R$ given by the union of the towers, with the $\sigma$-algebra generated by the levels of the towers.  In the case that $n_i\rightarrow\infty$ and the measure of the levels
of the towers goes to $0$, the resulting action $(X,\mu,\{T^{\vec v}\}_{\vec v\in\mathbb Z^2})$ is
a rank one $\mathbb Z^2$ action.  We say that $\{\tau_i\}$ is a sequence of towers associated to $T$.
As the levels of $\tau_i$ all have the
same measure, $T$ is clearly (possibly infinite) measure
preserving.

Note that just as we can see slices of $\tau_{i}$ inside $\tau_{i+1}$, we can 
find slices of those slices inside $\tau_{i+2}$, and so on.  
 
The relative positions of all the slices of $\tau_i$ inside later towers are important for our arguments.  To this end, for all $j\ge 1$ we define $\mathcal P_i(j)$ to be those vectors $\vec u\in B_{n_{i+j}}$ with the property that $T^{\vec u}$ maps the base of one slice of $\tau_i$ in $\tau_{i+j}$ to the base of another slice.  Let $\mathcal P_i=\cup_{j\ge1}\mathcal P_i(j)$.  We call the elements of $\mathcal P_i$ {\em times of strong recurrence} for the tower $\tau_i$.  The next lemma is immediate.

\begin{lemma}\label{wholetower}
Let $I = I_{\vec{v}}^{i}$ be a level of tower $\tau_{i}$.  Then
   $T^{\vec{u}} I \cap I \neq \emptyset$ if and only if $\vec{u}$ is a time
   of strong recurrence for tower $\tau_{i}$.
\end{lemma}

There are several well known properties of rank one actions that we will use extensively.  

\begin{proposition}\label{wellapprox}
Let $(X,\mu,\{T^{\vec v}\}_{\vec v\in\mathbb Z^2})$ be a rank one action, and $\{\tau_i\}$ an associated sequence of towers.  Given a measurable set $A$ with $0<\mu(A)<\infty$ and $\epsilon>0$
there is an $i_0$ such that for all $i>i_0$ there is a level $I$ of
$\tau_i$ with the property that
\begin{equation}\label{more}
\mu(A\cap I) >(1-\epsilon)\mu(I).
\end{equation}
\end{proposition}

\begin{theorem}
A rank one measure preserving $\mathbb Z^2$ action is ergodic and
recurrent.

\end{theorem}
\begin{proof}
Let $(X,\mu,\{T^{\vec v}\}_{\vec v\in\mathbb Z^2})$ be a measure preserving, rank one $\mathbb Z^2$ action with $\{\tau_i\}$ an associated sequence of towers.  A proof of the ergodicity of $T$ can be found in Theorem 4.2.1 of \cite{RudSilva}.
 To see that $T$ is a recurrent $\mathbb Z^2$ action fix $A\subset X$, a measurable set of finite positive measure, and $\epsilon$ where $0<\epsilon<\frac14$.  Apply  Proposition~\ref{wellapprox} to obtain an $i_0$  and choose $i>i_0$. Let $I$ be the level of $\tau_i$ satisfying \eqref{more}.  By construction, $\tau_{i+1}$ is created from slices of $\tau_i$: in doing so the level $I$ is divided into $k(i)$ equal measure subintervals, each belonging to a slice $\tau_i^j$ of $\tau_i$.  By a standard Fubini argument there must be at least two of these subintervals, call them $I^{j_1}$ and $I^{j_2}$, for which $\mu (A\cap I^{j_k}) > (1-\sqrt{\epsilon})\mu (I^{j_k}) > \frac{1}{2} \mu (I^{j_k})$ for $k=1,2$.  Let $\vec u\in\mathcal P_1(i)$ denote the vector that maps the base of $\tau_i^{j_1}$ to the base of $\tau_i^{j_2}$.  In particular, $T^{\vec u}I^{j_1}=I^{j_2}$.  We are then guaranteed that $\mu\left(T^{\vec u}A\cap A\right)>0$.
 
 \end{proof}

Finally note that to construct a rank one $\mathbb Z^2$ action using the procedure we describe above it suffices to establish a value for $n_1$ and then for each $i$ to choose the number $k(i)$ of slices of $\tau_i$ used to construct $\tau_{i+1}$, along with the set $\mathcal P_i(1)\subset \mathbb Z^2$ describing the location of these slices.  These parameters completely determine the values of the rest of the $n_i, i\ge 2$ and sets $\mathcal P_i$.

\subsection{Proof of Theorem~\ref{none}}
To start the construction let $n_1=2$, and construct $\tau_1$ from the levels $\{[0,1),[1,2),[2,3),[3,4)\}$ indexed by $B_2$ with $I_{\vec 0}^1=[0,1)$.  Set $k(i)=2$ for all $i$ and let $\mathcal P_1(1)=\{\vec 0,(2,2)\}$.  

To define the induction let $\vec v_i^*$ denote the non-zero element of all the times of strong recurrence that have thus far been created that has the smallest slope.  Call this slope $m_i^*$.   Choose a positive vector $\vec v\in\mathbb Z^2$ 
 
and $k\in\mathbb Z$ so that  if $k\vec v=(x,y)$ then
\begin{align}
y&>n_i\label{grow}\\
\frac{y}{x}&<\frac12m_i^*\label{lesssteep}\\
k\vec v&\text{ lies below the $\vec v_i^*$-tunnel of $B_{n_i}$, and}\label{below}\\
\intertext{if $\vec w$ is a time of strong recurrence that has been created up to this point then}
k\vec v +B_{n_i} &\text{ does not intersect the $\vec w$-tunnels of $B_{n_i}$.}\label{nointer}
\end{align}
Let $\mathcal P_i(1)=\{\vec 0,k\vec v\}$.

Let $(X,\mu,\{T^{\vec v}\}_{\vec v\in\mathbb Z^2})$ denote the ergodic, infinite measure preserving, conservative $\mathbb Z^2$ action constructed in this way.  To see that $\mathcal R_T=\emptyset$ first consider the case where there exists a rational direction $\theta$ with an associated vector $\vec v\in\mathbb Z^2$ so that $\vec v\in\mathcal P_i$ for some $i$.  Choose the first stage $i$ for which this is true and let $s$ denote a spacer level from this stage.  If there exists $n\in\mathbb Z$ such that $\mu(T^{n\vec v}(s)\cap s)>0$ then by Lemma~\ref{wholetower} one slice of $\tau_i$ must intersect the $\vec v$-tunnel of another slice of $\tau_i$ at some stage $k>i$.  By condition \eqref{nointer} this is impossible since it implies that one of the two slices of $\tau_{k-1}$ must then intersect the $\vec v$-tunnel of the other.

Now suppose $\theta$ does not have an associated vector in $\mathbb{Z}^2$ which is a time of strong recurrence.  If $\theta=0$, then if $\epsilon=1$, condition \eqref{grow} implies that all slices of $\tau_1$ save one have to lie above the $1$-tunnel of the horizontal axis.  
If $\theta\neq 0$, then by \eqref{lesssteep} we can find $i$ such that $\vec v_k^*$ has slope smaller than $\tan\theta$ for all $k\ge i$.  Let $s$ be a spacer level in $\tau_{i+1}$  and suppose there exists $\vec w\in\mathbb Z^2$ in the $\epsilon$-tunnel of $\theta$ for some $\epsilon>0$ such that $\mu(T^{\vec w}s\cap s)>0$.  Arguing as above this implies there exists a $k>i$ such that one slice of $\tau_k$ intersects the $\vec u(\theta)$-tunel of the other at stage $k+1$.  But by \eqref{below} and \eqref{nointer} this is impossible.

\subsection{Proof of Theorem~\ref{all}}
Let $\{\theta_j\}$ denote the set of rational directions.  We first describe the construction that will guarantee $\theta_j\in\mathcal R_T$ for all $j$.  Let $\mathcal U=\{u_i\}$ denote an enumeration of the elements of $\mathbb Z^2$.   
 The towers will be of shape $\overline B_n$ centered at the origin, and the construction can start with an arbitrary choice for $n_1$.  For $i>1$ the remaining parameters are chosen to be $k(i)=2$ and $\mathcal P_i(1)=\{\vec 0,\vec w_i\}$  with $\vec w_i=t_i\vec u_i$ where $t_i\in\mathbb Z$ is chosen so that $\vec w_i+\overline B_{n_i}$ is disjoint from $\overline B_{n_i}$.    

Let $(X,\mu,\{T^{\vec v}\}_{\vec v\in\mathbb Z^2})$ denote the resulting $\mathbb Z^2$ action.  To see that $\mathcal R_T$ contains all rational directions fix a rational $\theta_j$, a measurable set $A$ of finite positive measure, and $0<\epsilon<\frac14$ and apply Proposition~\ref{wellapprox} to obtain $i_0$.  Note that there are infinitely many vectors $\vec u_k\in\mathbb Z^2$ associated with $\theta_j$ and therefore there are infinitely many $i>i_0$ so that $\mathcal P_i(1)$ consists of $\vec 0$ and a vector associated with $\theta_j$.   For such an $i$ let $I$ be the level in $\tau^i$ satisfying \eqref{more}.  This implies that both slices of $I$ in $\tau^{i+1}$ must be strictly more than $\frac12$ covered by $A$ and therefore $\mu(T^{\vec u}A\cap A)>0$ for the non-zero vector $\vec u\in P_i(1)$ associated with $\theta_j$.  
 
It follows from Theorem~\ref{Gdelta} that $\mathcal R_T$ must contain other directions besides the $\theta_j$.  We now describe how to modify the construction so that we can guarantee the given collection of irrational directions $\alpha_1,\ldots,\alpha_k$ are not in $\mathcal R_T$.  First note that for any $\epsilon>0$ every time a vector $t\vec w$ associated to $\theta_k$ appears as a time of strong recurrence, it is in the $\epsilon$ tunnel for only the interval of directions $(\theta_k-\delta(\epsilon,t),\theta_k+\delta(\epsilon,t))$ where $\delta(\epsilon,t)=\arcsin\left(\frac{\epsilon}{\Vert t\vec w\Vert}\right)$.    
There is some flexibility in the construction to choose $t$ as large as we like, thereby shrinking this interval as much as we like.  We next describe how to use this flexibility to guarantee that none of the $\alpha_i$ are recurrent directions for $T$.  

Fix a sequence $\epsilon_i$ decreasing to zero and at each stage of the construction choose $t_i$ so that the interval $(\theta_{k(i)}-\delta(\epsilon_i,t),\theta_{k(i)}+\delta(\epsilon_i,t))$ does not contain any of the directions $\alpha_j$.   For any $\epsilon>0$ there exists $i_o$ such that for all $i\ge i_0$ $\epsilon_i<\epsilon$.  Let $I$ be a level of $\tau_{i_0}$.  By Lemma~\ref{wholetower} we have that for all $j$, $\alpha_j$ cannot be $\epsilon$ recurrent for $I$.

 \def\cprime{$'$}
\providecommand{\bysame}{\leavevmode\hbox to3em{\hrulefill}\thinspace}
\providecommand{\MR}{\relax\ifhmode\unskip\space\fi MR }
% \MRhref is called by the amsart/book/proc definition of \MR.
\providecommand{\MRhref}[2]{%
  \href{http://www.ams.org/mathscinet-getitem?mr=#1}{#2}
}
\providecommand{\href}[2]{#2}

\end{document}